\documentclass[12pt]{amsart}

\usepackage{setspace, amsmath, amsthm, amssymb, amsfonts, amscd, epic, graphicx, ulem, dsfont}
\usepackage[T1]{fontenc}
\usepackage{multirow}
\usepackage{bbm}

\makeatletter \@namedef{subjclassname@2010}{
    \textup{2020} Mathematics Subject Classification}
\makeatother

\newtheorem{thm}{Theorem}[section]
\newtheorem{lem}{Lemma}[section]
\newtheorem{pro}{Proposition}[section]
\newtheorem{cor}{Corollary}[section]

\theoremstyle{remark}
\newtheorem*{rema}{Remark}

\newtheorem{exa}{\textbf{Example}}

\theoremstyle{definition}
\newtheorem{defn}{Definition}[section]

\newcommand{\R}{\mathbb{R}}

\newcommand{\N}{\mathbb{N}}

\begin{document}

    \title[]{On Generalized Powers of Operators}
    \author[A. Bachir,  M. H. Mortad and A. S. Nawal ]{Ahmed Bachir, Mohammed Hichem Mortad$^*$ and Nawal Ali Sayyaf}

    \thanks{}

    \dedicatory{}
    \thanks{* Corresponding author.}
    \date{}


    \address{(The first author) Department of Mathematics, College of Science, King Khalid University, Abha, Saudi
        Arabia.}

    \email{abishr@kku.edu.sa, bachir1960@icloud.com}

    \address{(The corresponding author) Department of
        Mathematics, University of Oran 1, Ahmed Ben Bella, B.P. 1524, El
        Menouar, Oran 31000, Algeria.}

    \email{mhmortad@gmail.com,  mortad.hichem@univ-oran1.dz.}

    \address{(The third author) Department of
        Mathematics, University of Bisha, Bisha, Saudi Arabia}

    \email{nasayyaf@ub.edu.sa}
    \begin{abstract}
        In this note, we introduce generalized powers of linear operators.
        More precisely, operators are not raised to numbers but to other
        operators. We discuss several properties as regards this notion.
    \end{abstract}

    \maketitle

    \section{Classical exponentials and logarithms of bounded linear operators}

    Let $H$ be a complex Hilbert space and let $B(H)$ be the algebra of
    all bounded linear operators defined from $H$ into $H$. Most of the
    notions and definitions used here are as in
    \cite{Mortad-Oper-TH-BOOK-WSPC}.

    This section is mainly devoted to basic properties of exponentials
    and logarithms of bounded operators. Most of these results should be
    known to readers. Most of the material here is borrowed from
    \cite{Mortad-Oper-TH-BOOK-WSPC}. In the end, readers are expected to
    have a certain knowledge on the continuous functional calculus.

    Let $A\in B(H)$. It is known that the series
    $\sum_{n=0}^{\infty}{A^n}/{n!}$ converges absolutely in $B(H)$, and
    hence it converges. This allows us to define $e^A$ (where $A\in
    B(H)$) without using all the theory of the functional calculus.

    It is easy to see that $e^A$ is self-adjoint whenever $A$ is. The
    converse, however, is not always true as witnessed by $A=2i\pi I$.
    Readers interested in when the converse holds may consult
    \cite{Kurepa-LOG-1962}, \cite{Putnam-LOG-S.A}, and
    \cite{Schmoeger-DEM-MATh-2002}.

    As regards logarithms of bounded operators, there are different ways
    for defining them. One possible way is to use the series
    \[\log A:=\sum_{n=1}^{\infty}\frac{(-1)^{n-1}}{n}(A-I)^n\]
    which converges in $B(H)$ whenever $A\in B(H)$ is such that
    $\|A-I\|<1$. Another way is to call any $A\in B(H)$ such as $e^A=B$,
    where $B\in B(H)$ is given, as a logarithm of $B$. However, we
    choose to use the continuous functional calculus to define logarithm
    of operators.

    Let $A\in B(H)$ be positive and invertible, and so $\sigma(A)\subset
    (0,\infty)$. Hence the function $\log$ is well-defined on
    $\sigma(A)$. Therefore, it also makes sense to define $\log(A)$. We
    call it the logarithm of $A$. It is also clear why $\log(A)$ is
    self-adjoint.

    \begin{exa}
        Let $A\in B(H)$ be self-adjoint. Then
        \[\log(e^A)=A.\]
        In particular,
        \[\log I=\log(e^0)=0.\]
        If we further assume that $A$ is positive with $0\not\in\sigma(A)$
        (i.e. if $A$ is invertible), then
        \[e^{\log A}=A.\]
    \end{exa}

    \begin{thm}
        Let $A,B\in B(H)$ be self-adjoint. Then
        \[AB=BA\Longleftrightarrow e^Ae^B=e^Be^A.\]
    \end{thm}

    The previous theorem is well-known (cf. \cite{wermuth.1997.pams}),
    but next we give an elementary proof.

    \begin{proof}(\cite{Mortad-Oper-TH-BOOK-WSPC}) Let $f(x)=\ln x$ be defined for $x>0$.

        We only prove "$\Leftarrow$". Assume that $e^Ae^B=e^{B}e^A$. Since
        $e^A$ is positive, we may write
        \[f(e^A)e^B=e^Bf(e^A) \text{ or merely } \ln(e^A)e^B=e^B\ln(e^A)\]
        and so
        \[Ae^B=e^BA.\]
        Since $B$ is self-adjoint too, the same reasoning mutatis mutandis
        gives
        \[AB=BA.\]
    \end{proof}

    We may recover some known properties of "$\log$". For instance, we
    have:

    \begin{pro}\label{log A log AB= log A+log B}
        Let $A,B\in B(H)$ be both positive and such that
        $0\not\in\sigma(A)\cap\sigma(B)$. Then
        \[AB=BA\Longrightarrow \log(AB)=\log A+\log B.\]
        In particular,
        \[\log(A^{-1})=-\log A.\]
    \end{pro}

    \begin{proof} First, $AB$ is positive as by assumption $A,B\geq0$ and $AB=BA$. Since
        $A$ and $B$ are invertible, it follows that $AB$ too is invertible.
        Therefore, the quantities $\log(AB)$, $\log A$ and $\log B$ are all
        well defined. Notice also that
        \[AB=BA\Longrightarrow \log A\log B=\log B\log A.\]
        We are ready to prove the desired equality. We have
        \[e^{\log A+\log B}=e^{\log A}e^{\log B}=AB.\]
        Therefore (and since $\log A+\log B$ is self-adjoint),
        \[\log(AB)=\log(e^{\log A+\log B})=\log A+\log B,\]
        as desired.

        Finally, since $A$ is invertible, $AA^{-1}=A^{-1}A=I$. Besides,
        $A^{-1}$ is positive and invertible. Hence
        \[0=\log I=\log A+\log A^{-1}\]
        and so
        \[\log(A^{-1})=-\log A,\]
        as required.

    \end{proof}

    \section{Generalized powers of operators}

    We are ready to introduce the concept of generalized powers of
    operators.

    \begin{defn}
        Let $A,B\in B(H)$ be such that $A$ is positive and invertible.
        Define
        \[A^B=e^{B\log A}.\]
        Then $B$ is called the generalized power of $A$.

        Say that $A$ is a root of order $B$ of $T\in B(H)$ provided
        \[A^B=T.\]
    \end{defn}

    \begin{rema}
        Due to the absence of commutativity in $B(H)$, $A^B$ cannot be
        defined here as $e^{(\log A)B}$.
    \end{rema}

    \begin{rema}
        The preceding definition does generalize the usual definition of
        (ordinary) powers of operators. For instance, if we confine our
        attention to natural powers, then by setting $B=n I$, where
        $n\in\N$, we get back to the usual definition of $A^n$. Indeed
        \[A^{nI}=e^{nI\log A}=e^{n\log A}=e^{\log A^n}=A^n,\]
        as wished.
    \end{rema}

    \begin{exa}Let us find $A^B$ where
        \[A=\left(
        \begin{array}{cc}
            1 & 0 \\
            0 & 2 \\
        \end{array}
        \right)
        \text{ and } B=\left(
        \begin{array}{cc}
            0 & 1 \\
            0 & 0 \\
        \end{array}
        \right).\]
        Then
        \[B\log A=\left(
        \begin{array}{cc}
            0 & \log2 \\
            0 & 0 \\
        \end{array}
        \right).\]
        Thus,
        \[\left(
        \begin{array}{cc}
            1 & 0 \\
            0 & 2 \\
        \end{array}
        \right)^{\left(
            \begin{array}{cc}
                0 & 1 \\
                0 & 0 \\
            \end{array}
            \right)}=\left(
        \begin{array}{cc}
            1 & \log2 \\
            0 & 1 \\
        \end{array}
        \right)\]
    \end{exa}

    The first natural question is to get the value or an estimate of the
    norm of $A^B$ in terms of the norms of $A$ and $B$.

    \begin{pro}\label{A ouss B NORM FIRST PROPOSITION!}
        Let $A,B\in B(H)$ be such that $A$ is positive and invertible. Then
        $A^B\in B(H)$ and
        \[\|A^B\|\leq e^{\|B\log A\|}.\]
        If we further assume that $AB=BA$ and that $B$ is positive, then
        \[\|A^B\|=e^{\|B\log A\|}.\]
    \end{pro}

    The proof relies on the following perhaps known result. We include a
    proof for readers convenience.

    \begin{lem}\label{norm eA WSPC LEMMA}
        Let $T\in B(H)$ be positive. Then
        \[\|e^T\|=e^{\|T\|}\]
    \end{lem}

    \begin{proof}
        We already know by the general theory that $\|e^T\|\leq e^{\|T\|}$.
        To show the reverse inequality, first recall that $\|T\|\in
        \sigma(T)$. Using both the spectral radius theorem and the spectral
        mapping theorem yield
        \[\|e^T\|=\sup\{|\lambda|:\lambda\in \sigma(e^T)\}=\sup\{e^{\mu}:~\mu\in\sigma(T)\}\geq e^{\|T\|},\]
        as needed.
    \end{proof}

    Now, we show Proposition \ref{A ouss B NORM FIRST PROPOSITION!}.

    \begin{proof}
        The first inequality is clear. To show the full equality in the
        second case, merely observe that $AB=BA$ gives $B\log A=(\log A)B$.
        In other words, $B\log A$ is positive. By Lemma \ref{norm eA WSPC
            LEMMA}, we obtain
        \[\|A^B\|=\|e^{B\log A}\|=e^{\|B\log A\|},\]
        as wished.
    \end{proof}

    \begin{rema}
        Instead of assuming that $AB=BA$, we could have assumed that $AB$ is
        normal. Indeed, since $A$ is positive, $AB$ is self-adjoint by say
        \cite{MHM1}, that is, $AB=BA$. By \cite{Dehimi-Mortad-INVERT} or
        else,
        \[\sigma(AB)=\sigma(\sqrt BA\sqrt B)\subset \R^+\]
        which forces $AB$ to be positive. This remark applies to some of the
        results below as well.
    \end{rema}

    The next results give more basic operations about this new notion.

    \begin{pro}
        Let $A\in B(H)$ be positive and invertible. Then $A^B$ is positive
        and invertible for any $B\in B(H)$ such that $AB=BA$. Moreover,
        \[\log(A^B)=B\log A.\]
    \end{pro}

    \begin{proof}Since $AB=BA$, it follows $B\log A=(\log A)B$, from which
        we derive the self-adjointness of $B\log A$. Since the latter is
        self-adjoint, $e^{B\log A}$ or $A^B$ is positive. To show the second
        claim, just observe that
        \[\log(A^B)=\log(e^{B\log A})=B\log A.\]
    \end{proof}

    \begin{thm}\label{jklopiieozopapazpzppzp} Let $A,B,C\in B(H)$ and let $I$ be the usual identity
        operator on $H$. Assume that $A$ is positive and invertible. Then
        \begin{enumerate}
            \item $I^B=I$, $A^0=I$ and $A^I=A$.
            \item $(AB)^C=A^CB^C$ whenever $A$, $B$ and $C$ pairwise commute, and $B$ is positive and invertible.
            \item $(A^B)^C=A^{CB}$.
            \item $A^BA^C=A^{B+C}$ whenever $A$, $B$ and $C$ pairwise commute.
        \end{enumerate}
    \end{thm}

    \begin{proof}\hfill
        \begin{enumerate}
            \item Clearly
            \[I^B=e^{B\log I}=e^{B0}=e^0=I.\]
            The identities  $A^0=I$ and $A^I=A$ are even clearer.
            \item We have
            \[(AB)^C=e^{C\log(AB)}=e^{C(\log A+\log B)}=e^{C\log A+C\log B}.
            \]
            Since $CA=AC$, it ensues that $C\log A=(\log A) C$. Similarly,
            $C\log B=(\log B) C$. Hence $C\log A$ commutes with $C\log B$ as
            $\log A\log B=\log B\log A$. Therefore,
            \[e^{C\log A+C\log B}=e^{C\log A}e^{C\log B}=A^CB^C,\] and so
            \[(AB)^C=A^CB^C,\]
            as needed.
            \item We may write
            \[(A^B)^C=e^{C\log(A^B)}=e^{CB\log A}=A^{CB}.\]
            \item First, recall that $B\log A$ commutes with $C\log A$. Then, we have
            \[A^{B+C}=e^{(B+C)\log A}=e^{B\log A+C\log A}=e^{B\log A}e^{C\log A}=A^BA^C,\]
            as suggested.
        \end{enumerate}
    \end{proof}

    \begin{cor}
        Let $A,B\in B(H)$ be both positive and invertible operators such
        that $AB=BA$. Then $A^TB^T=B^TA^T$ for any $T\in B(H)$ which
        commutes with both $A$ and $B$.
    \end{cor}

    \begin{proof}Since $A$, $B$ and $T$ all pairwise commute, the
        foregoing theorem yields
        \[A^TB^T=(AB)^T=(BA)^T=B^TA^T,\]
        as needed.
    \end{proof}

    Next, we give results involving the notion of the adjoint.

    \begin{thm}
        Let $A\in B(H)$ be positive and invertible. Let $B\in B(H)$ be such
        that $AB=BA$. Then
        \begin{enumerate}
            \item $(A^B)^*=A^{B^*}$. In particular, when $B$ is self-adjoint, then $A^B$ is also self-adjoint.
            \item If $B$ is normal, then $A^B$ is also normal.
            \item If $B$ is anti-symmetric, i.e. $B^*=-B$, then $A^B$ is unitary.
        \end{enumerate}
    \end{thm}

    \begin{proof}Suppose $AB=BA$. Then $(\log A)B=B\log A$ from which we
        derive $(\log A)B^*=B^*\log A$. Then
        \[(A^B)^*=\left(e^{B\log A}\right)^*=e^{(B\log A)^*}=e^{(\log A)B^*}=e^{B^*\log A}=A^{B^*}.\]

        If $B$ is normal, then $BB^*=B^*B$, and since $A$ commutes with
        $B^*$, Theorem \ref{jklopiieozopapazpzppzp} gives
        \[(A^B)^*A^B=A^{B^*}A^B=A^{B^*+B}\]
        and
        \[A^B(A^B)^*=A^BA^{B^*}=A^{B+B^*}.\]
        Therefore, $A^B$ is normal.

        Finally, if $B^*=-B$, then $B$ is normal and so
        \[(A^B)^*A^B=A^B(A^B)^*=A^{B+B^*}.\]
        Because $B^*=-B$, $A^{B+B^*}=A^0=I$. Consequently, $A^B$ is unitary.
    \end{proof}

    The converse of the previous proposition is, in general, untrue. To
    see why, it suffices to consider $I^B$ (indeed, $I^B=I$ for any
    $B\in B(H)$). Imposing a certain control on the norm of $B\log A$,
    however, resolves the issue.

    \begin{pro}Let $A\in B(H)$ be positive and invertible. Let $B\in B(H)$ be
        self-adjoint such that $A^B$ is also self-adjoint. If $\|B\log
        A\|<2\pi$, then $AB=BA$.
    \end{pro}

    \begin{proof}Since $\|B\log
        A\|<2\pi$ and $A^B=e^{B\log A}$ is self-adjoint, by
        \cite{Kurepa-LOG-1962} (cf. \cite{Schmoeger-DEM-MATh-2002}), we
        infer that $B\log A$ is self-adjoint. In other words, $B\log A=(\log
        A) B$. Thus,
        \[BA=Be^{\log A}=e^{\log A} B=AB,\]
        as needed.

    \end{proof}

    Let us now examine the validity of the well-known Heinz inequality.
    Recall that this inequality says that if $A,B\in B(H)$ are such that
    $0\leq A\leq B$, then $A^\alpha\leq B^\alpha$ for any
    $\alpha\in[0,1]$. See \cite{FUR.book} for a proof. The inequality
    remains valid for $\alpha>1$ if the condition $AB=BA$ is added.

    \begin{pro}
        Let $A,B\in B(H)$ be both positive such that $A\leq B$ and $A$ is
        invertible. Then
        \[A^T\leq B^T\]
        for any positive $T\in B(H)$ provided $T$, $A$ and $B$ pairwise
        commute.
    \end{pro}

    \begin{proof}First, the invertibility of $B$, which is essential, is
        guaranteed by the inequality $A\leq B$ (see Exercise 5.3.30 in
        \cite{Mortad-Oper-TH-BOOK-WSPC} for a new proof of this known fact).

        Since $0\leq A\leq B$ and both $A$ and $B$ are invertible, and as a
        consequence of the classical Heinz inequality, we know that $\log
        A\leq \log B$. Since $T$ is positive and commutes with $A$, it
        ensues that $T$ commutes with $\log A$. Similarly, $T$ commutes with
        $\log B$. Therefore
        \[T\log A\leq T\log B.\]
        Since $T\log A$ commutes with $T\log B$, it follows that by say
        Exercise 9 on Page 243 in \cite{Con} that
        \[e^{T\log A}\leq e^{T\log B}\]
        or
        \[A^T\leq B^T,\]
        as wished.
    \end{proof}

    \begin{rema}
        It is unknown to me whether the previous result holds if the
        commutativity of $A$ and $B$ is dropped.
    \end{rema}

    \begin{rema}
        Readers are probably already aware that there are self-adjoint
        matrices $A$ and $B$ such that $A\geq B$ but $e^A\not\geq e^B$ (see
        e.g. \cite{Mortad-counterexamples book OP TH} for a counterexample).
    \end{rema}

    In the end, let us give a couple of examples or so.

    \section{Conclusion}
    This new concept should hopefully stimulate further research as
    regards many already known results which involve classical powers of
    operators.

    One is also tempted to define a similar notion for the class of
    unbounded operators. As it is known, we will have some domain issues
    to fix. Perhaps, this would constitute a follow-up paper. Let me say
    a word or two about possible extensions. Let $A$ be an unbounded
    self-adjoint positive operator which is also boundedly invertible,
    and let $B\in B(H)$ be self-adjoint for the moment. Thanks to the
    functional calculus for unbounded self-adjoint operators (see e.g.
    \cite{SCHMUDG-book-2012}), we may define $\log A$. Now, if we set
    \[A^B=e^{B\log A},\]
    then it is seen that $B\log A$ is not necessarily self-adjoint (it
    is not even normal). Indeed, it is not necessarily a closed
    operator. So there are a few modifications one can carry out to
    obtain a well defined formula. For instance, we could assume that
    $B$ is invertible or set
    \[A^B=e^{\overline{B\log A}}\text{ or }A^B=e^{(\log A)B}\]
    where $\overline{B\log A}$ denotes the closure of the operator
    $B\log A$. This is still non-sufficient to have a well defined
    formula. Again, there are self-adjoint operators $T$ (unbounded) and
    $S\in B(H)$ such that $D(TS)=\{0\}$ (see e.g.
    \cite{Mortad-counterexamples book OP TH} or \cite{Mortad-34page
        paper square roots et al.}). In other words, $(\log A)B$ could only
    be defined on $\{0\}$.

    So, we need to also assume the commutativity condition $BA\subset
    AB$ so that $B\log A\subset (\log A) B$ by the spectral theorem (see
    e.g. Theorem 4.11 on Page 323 in \cite{Con}). This commutativity
    assumption now guarantees the self-adjointness of $(\log A) B$ (or
    its normality if $B$ is only normal). Some very related papers are
    \cite{Gustafson-Mortad-II}, \cite{Jung-Mortad-Stochel},
    \cite{Meziane-Mortad-I}, \cite{MHM1}, \cite{Mortad-OaM-2014}, and
    when dealing with sums of generalized powers we would need
    \cite{Mortad-Normality-Sum} and \cite{Mortad-2015-RCSP}.

    \bibliographystyle{amsplain}

\begin{thebibliography}{1}

        \bibitem{Chaban-Mortad-2013-exp}
        A. Chaban, M. H. Mortad. Exponentials of bounded normal operators,
        \textit{Colloq. Math.,} {\bf 133/2} (2013) 237-244.

        \bibitem{Con}
        J. B. Conway. A course in functional analysis, \textit{Springer},
        1990 (2nd edition).

        \bibitem{Dehimi-Mortad-INVERT}
        S. Dehimi, M. H. Mortad. Right (or left) invertibility of bounded
        and unbounded operators and applications to the spectrum of
        products, \textit{Complex Anal. Oper. Theory}, \textbf{12/3} (2018)
        589-597.


        \bibitem{FUR.book}
        T. Furuta. Invitation to Linear Operators: From Matrices to Bounded
        Linear Operators on a Hilbert Space, \textit{Taylor \& Francis
            Group, London}, 2001.

        \bibitem{Gustafson-Mortad-I}
        K. Gustafson, M. H. Mortad. Unbounded products of operators and
        connections to Dirac-type operators,  \textit{Bull. Sci. Math.},
        \textbf{138/5} (2014), 626-642.

        \bibitem{Gustafson-Mortad-II}
        K. Gustafson, M. H. Mortad. Conditions implying commutativity of
        unbounded self-adjoint operators and related topics, \textit{J.
            Operator Theory,} \textbf{76/1}, (2016) 159-169.

        \bibitem{halmos-book-1982}
        P. R. Halmos. \textit{ A Hilbert space problem book}, Springer, 1982
        (2nd edition).

        \bibitem{Jung-Mortad-Stochel}
        Il B. Jung, M. H. Mortad, J. Stochel. On normal products of
        selfadjoint operators, \textit{Kyungpook Math. J.}, \textbf{57}
        (2017) 457-471.

        \bibitem{Kubrusly-book-operatopr-exercise-sol}
        C. S. Kubrusly. \textit{Hilbert space operators, A problem solving
            approach}, Birkh\"{a}user. Boston, Inc., Boston, MA, 2003.

        \bibitem{Kubrusly-2011-COURS+EXO-FAT-BOOK}
        C. S. Kubrusly. \textit{The elements of operator theory}.
        Birkh\"{a}user/Springer, New York, 2011 (2nd edition).

        \bibitem{Kurepa-LOG-1962}
        S. Kurepa. A note on logarithms of normal operators, \textit{Proc.
            Amer. Math. Soc.}, \textbf{13} (1962) 307-311.

        \bibitem{Meziane-Mortad-I}
        M. Meziane, M. H. Mortad. Maximality of linear operators,
        \textit{Rend. Circ. Mat. Palermo, Ser II.}, \textbf{68/3} (2019)
        441-451.

        \bibitem{MHM1}
        M. H. Mortad. An application of the Putnam-Fuglede theorem to normal
        products of self-adjoint operators, \textit{Proc. Amer. Math. Soc.},
        {\bf 131/10}, (2003) 3135-3141.

        \bibitem{Mortad-Scattering} M. H. Mortad. Explicit formulae for the wave
        operators of perturbed self-adjoint operators, \textit{J. Math.
            Anal. Appl.}, {\bf 356/2}, (2009) 704-710.

        \bibitem{Mortad-Exp-normal-Colloquium}
        M. H. Mortad. Exponentials of normal operators and commutativity of
        operators: A new approach, \textit{Colloq. Math.}, {\bf 125/1}
        (2011) 1-6.

        \bibitem{Mortad-Normality-Sum}
        M. H. Mortad. On the normality of the sum of two normal operators,
        \textit{Complex Anal. Oper. Theory}, {\bf 6/1} (2012) 105-112.

        \bibitem{Mortad-OaM-2014}
        M. H. Mortad. Commutativity of unbounded normal and self-adjoint
        operators and applications, \textit{Oper. Matrices}, \textbf{8/2}
        (2014) 563-571.

        \bibitem{Mortad-2015-RCSP}
        M.H. Mortad. A criterion for the normality of unbounded operators
        and applications to self-adjointness, \textit{Rend. Circ. Mat.
            Palermo (2)}, \textbf{64/1} (2015) 149-156.

        \bibitem{Mortad-Oper-TH-BOOK-WSPC}
        M. H. Mortad. \textit{An operator theory problem book}, World
        Scientific Publishing Co., (2018).


        \bibitem{Mortad-counterexamples book OP TH}
        M. H. Mortad, \textit{Counterexamples in operator theory}, (book, to
        appear). Birkh\"{a}user/Springer.

        \bibitem{Mortad-34page paper square roots et al.}
        M. H. Mortad. Unbounded operators: (square) roots, nilpotence,
        closability and some related invertibility results. arXiv:2007.12027

        \bibitem{Putnam-LOG-S.A}
        C. R. Putnam. On square roots and logarithms of self-adjoint
        operators, \textit{Proc. Glasgow Math. Assoc.}, \textbf{4} (1958),
        \textbf{1-2} (1958).

        \bibitem{RS1}
        M. Reed, B. Simon. Methods of modern mathematical physics, Vol. {\bf
            1}: \textit{Functional analysis}, Academic Press. 1972.

        \bibitem{Schmoeger-DEM-MATh-2002}
        Ch. Schmoeger. On logarithms of linear operators on Hilbert spaces,
        \textit{Demonstratio Math.}, \textbf{35/2} (2002) 375-384.

        \bibitem{SCHMUDG-book-2012}
        K. Schm\"{u}dgen. \textit{Unbounded self-adjoint operators on
            Hilbert space}, Springer. GTM {\bf 265}  (2012).

        \bibitem{wermuth.1997.pams}
        E. M. E. Wermuth, A remark on commuting operator exponentials,
        \textit{Proc. Amer. Math. Soc.}, {\bf 125/6}, (1997) 1685-1688.

    \end{thebibliography}

\end{document}